\documentclass{amsart}
\usepackage{amssymb}
\usepackage{color}

\newtheorem{theorem}{Theorem}[section]
\newtheorem{lemma}[theorem]{Lemma}
\newtheorem{corollary}[theorem]{Corollary}
\newtheorem{proposition}[theorem]{Proposition}

\theoremstyle{remark}
\newtheorem{example}{Example}
\newtheorem{remark}{Remark}

\title{Precompact
noncompact reflexive abelian groups}

\author{S.~Ardanza-Trevijano, M.\,J.~Chasco, X.~Dom\'{\i}nguez,
and M.\,G.~Tkachenko}

\begin{document}
\maketitle

\begin{abstract}
We present a series of examples of precompact, noncompact, reflexive
topological Abelian groups. Some of them are pseudocompact or even
countably compact, but we show that there exist precompact
non-pseudo\-compact reflexive groups as well. It is also proved that
every pseudocompact Abelian group is a quotient of a reflexive
pseudocompact group with respect to a closed 
reflexive pseudocompact subgroup.

2010 Mathematics Subject Classification: 43A40, 22A05, 54H11 (primary), 54D30, 54A35 (secondary).
\end{abstract}

\section{Introduction}
\footnotetext{The authors were supported by  CONACyT of Mexico, grant
000000000074468; MEC of Spain BFM2006-03036, and FEDER funds.}
The Pontryagin--van Kampen duality theorem states that if $G$ is a
LCA (locally compact Abelian) group then the canonical evaluation
mapping $\alpha_G: G \to G^{\wedge\wedge}$ defined by
$\alpha_G(x)(\chi)=\chi(x)$ for all $x\in G$ and $\chi\in G^\wedge$
is a topological isomorphism. Here $G^\wedge$ denotes the group of
continuous characters $\chi: G\rightarrow{\mathbb T}$, where
$\mathbb{T}=\{z\in\mathbb{C}: |z|=1\}$ is the circle group in the
complex plane $\mathbb{C}$, and the groups $G^\wedge$ and $G^{\wedge
\wedge}$ carry, as usual, the compact-open topology $\tau_{co}$.

There have been many efforts to extend the Pontryagin--van Kampen
duality theorem outside the realm of locally compact Abelian groups.
The groups $G$ such that $\alpha_G$ is a topological isomorphism are
usually called \emph{reflexive}. It turns out that many non-locally
compact topological Abelian groups are reflexive. The first examples
were provided by Kaplan in \cite{Kaplan}, where he proved that any
product of reflexive groups is reflexive. The main motivation of
this paper is to answer the following question posed by M.~J.~Chasco
and E.~Mart\'{\i}n Peinador in \cite[p.~641]{ChM}: Is a precompact
reflexive Abelian group necessarily compact? We
recall that a precompact group is just a subgroup of a compact
group. Note that, since the dual of a metrizable precompact group is
discrete (see \cite[4.10]{Au} or \cite{Ch}), reflexive, precompact,
noncompact groups can only be found within non-metrizable groups.

 A topological space is said to be
\textit{pseudocompact\/} if every real-valued continuous function
defined on it is bounded. According to \cite[Theorem~1.1]{CR2},
pseudocompact groups form a proper subclass of the class of
precompact groups; actually they can be characterized as the
$G_{\delta}$-dense subgroups of compact groups. Being pseudocompact
yields nice relations between a topological Abelian group and its
dual. For example, Hern\'{a}ndez and Macario established in
\cite[Proposition~3.4]{HM} that a precompact Abelian group $G$ is
pseudocompact iff for every countable subgroup $H$ of $G^\wedge$,
the largest precompact topological group topology of $H$ coincides
with the topology $\omega(H,G)$ of pointwise convergence on elements
of $G$.

We show in Theorem~\ref{fincompreflex} that every pseudocompact
Abelian group without infinite compact subsets is reflexive. To
establish the existence of infinite pseudocompact Abelian groups
without infinite compact subsets, we use the notion of
\textit{$h$-embedded\/} subgroup of a topological group introduced
in \cite{Tk1}. Let $\mathcal{P}_h$ be the class of topological
Abelian groups $G$ which are pseudocompact and have the property
that every countable subgroup $C$ of $G$ is $h$-embedded, i.~e.,
every (not necessarily continuous) homomorphism from $C$ to the
circle group ${\mathbb T}$ admits an extension to a
\textit{continuous\/} homomorphism from $G$ to ${\mathbb T}$.
According to \cite{Tk1}, the class $\mathcal{P}_h$ contains many
infinite topological groups. We show in Proposition~\ref{Pro:1} that
all compact subsets of an arbitrary group $G\in\mathcal{P}_h$ are
finite.

Slightly refining our techniques, we present in Theorem~\ref{Th:4}
an example of a reflexive, precompact, non-pseudocompact group.
Under Martin's Axiom, a countably compact noncompact reflexive
group is presented in Proposition~\ref{Exa:HJ}. We show in
Example~\ref{negative2}, however, that $\omega$-bounded groups need
not be reflexive.

In Section~\ref{Sec:ClP} we give some insight into the wideness and
permanence properties of the class of reflexive pseudocompact
groups. We prove in Theorem~\ref{Th:QG} that \textit{every\/}
pseudocompact Abelian group is a quotient of a reflexive group
$H\in\mathcal{P}_h$ with respect to a closed pseudocompact subgroup
$L$. In particular, the group $L$ is reflexive. We
also answer a question in \cite{HM} by showing that there exists an
infinite reflexive pseudocompact group $G$ such that the dual group
$G^\wedge$ is also pseudocompact and, hence, fails to be a
$\mu$-space.

\subsection*{Notation and preliminary facts}
We only consider Abelian groups. The \emph{polar\/} of a subset $A$
of a topological group $G$ is the set
\[
A^{\vartriangleright}=\{\chi\in G^{\wedge}\,\colon\, \chi(A)\subset
\mathbb{T}_+\},
\]
where $\mathbb{T}_+=\{z\in\mathbb{T}:{\rm Re} z\ge 0\}$. The
\emph{inverse} polar of a set $B\subset G^{\wedge}$ is the set
\[B^{\vartriangleleft}=\{x\in G\,\colon\, \chi(x)\in \mathbb{T}_+\, \
\mbox{for\ each}\ \chi\in B\}.\]

Polars can be used to describe both the compact-open topology
$\tau_{co}$ and the topology of pointwise convergence, which we
denote by $\tau_p$. Polars of compact sets form a neighborhood basis
at the identity of $(G^\wedge, \tau_{co})$, while polars of finite
sets play the same role in $(G^\wedge, \tau_p)$. For a topological
Abelian group $G$, a set $E\subset G^\wedge$ is
\emph{equicontinuous\/} if there exists a neighborhood $U$ of the
neutral element in $G$ such that $E\subset U^\triangleright$.

The evaluation mapping $\alpha_G$ is continuous if and only if the
compact subsets of $(G^\wedge,\tau_{co})$ are equicontinuous
\cite[Proposition~5.10]{Au}. We will see below that this result has
strong implications for precompact groups.

For an Abelian group $A$, we denote by $Hom(A,\mathbb{T})$ the group
of all homomorphisms $h\colon A\to\mathbb{T}$ with the pointwise
multiplication. Given a subgroup $L$ of the group
$Hom(A,\mathbb{T})$, we will denote by $\omega(A,L)$ the topology on
$A$ induced by the elements of $L$. Comfort and Ross showed in their
seminal paper \cite{CR1} that the topology of a precompact Abelian
group $G$ always coincides with $\omega(G,G^{\wedge}).$ In
particular $\omega(A,Hom(A,{\mathbb T}))$ is the maximal precompact
topological group topology on a given abstract Abelian group $A$.
This topology is also known as the {\it Bohr topology} on $A$.

A topological group $G$ is \textit{sequentially complete\/} if it is
\textit{sequentially closed\/} in any topological group $H$ that
contains $G$ as a topological subgroup. In other words, no sequence
of elements of $G$ converges to an element of $H\setminus G$. It is easy
to see that $G$ is sequentially complete iff $G$ is sequentially
closed in its completion $\varrho{G}$.

Precompact groups are subgroups of compact groups, which in turn
satisfy the Pontryagin-van Kampen duality theorem. Hence it is
easy to deduce that for each precompact group $G$, $\alpha_G$ is
injective and open when considered as a mapping onto its image.
However, for a precompact group $G$, $\alpha_G$ can fail to be
continuous or surjective, as we see in the following examples.

\begin{example}\label{negative1}
Let $H$ be any locally compact, noncompact Abelian group. Denote by
$G$ the group $H$ endowed with the precompact topology
$\omega(H,H^\wedge)$. It is clear that $\omega(H,H^\wedge)$ is
strictly weaker than the original topology of $H$.

We claim that $\alpha_G$ is surjective but not continuous. Indeed,
$G$ has the same continuous characters as $H$, and it is a
consequence of the classical Glicksberg theorem in \cite{glicksberg}
that they have the same compact sets as well. Hence
$G^\wedge=(G^{\wedge},\tau_{co})$ is topologically isomorphic to
$H^\wedge$ and, since $H$ is reflexive, $G^{\wedge\wedge}\cong H$.
This implies our claim.
\end{example}

\begin{example}\label{negative2}
Let $I$ be an arbitrary uncountable index set and $n\ge 2$ a natural
number. For every point $x\in {\mathbb Z}(n)^I$, let ${\rm supp}(x)$
be the set of those $i\in I$ for which $x(i)\not =0$. Then
\[
\Sigma=\{x\in{\mathbb Z}(n)^I: |{\rm supp}(x)|\leq\omega\}
\]
is a dense $\omega$-bounded subgroup of the compact topological group
${\mathbb Z}(n)^I$ (see \cite[Corollary~1.6.34]{AT}). It is shown in
\cite{ChM} that the evaluation mapping $\alpha_{\Sigma}$  is
continuous but not surjective. In fact, $\Sigma^{\wedge}$ is the
direct sum ${\mathbb Z}(n)^{(I)}$ with the discrete topology, and hence the
bidual group $\Sigma^{\wedge \wedge}$ is the full product
${\mathbb Z}(n)^{I}$ with the Tychonoff product topology.
\end{example}

Let us say that a subgroup $D$ of a topological Abelian group $G$ is
\textit{$h$-embedded\/} in $G$ if every (not necessarily continuous)
homomorphism $f$ of $D$ to a compact Abelian topological group $K$
can be extended to a \textit{continuous\/} homomorphism
$\tilde{f}\colon G\to K$. This is equivalent to saying that every
homomorphism $f\colon D\to{\mathbb T}$ can be extended to a
\textit{continuous\/} character $\tilde{f}\colon G\to{\mathbb T}$.
Note that a subgroup $D$ of a precompact Abelian group $G$ is
$h$-embedded in $G$ if and only if the topology induced on $D$ by
$G$ coincides with the Bohr topology of $D$.

It is immediate from the definition that all homomorphisms of an
$h$-embedded subgroup $D$ of a topological Abelian group $G$ to
compact Abelian topological groups are continuous. In fact, since
precompact topological groups are subgroups of compact groups, every
homomorphism of such a subgroup $D\subset G$ to a precompact Abelian
group is continuous. Discrete subgroups of topological groups are
good candidates to be $h$-embedded in the enveloping groups, but an
$h$-embedded subgroup need not be discrete, even if the enveloping
group is compact.

\section{Reflexivity of pseudocompact groups}
In this section we present a sufficient condition for reflexivity of
pseudocompact groups. It turns out that the absence of infinite
compact subsets does the job (Theorem~\ref{fincompreflex}).

\begin{proposition}\label{Pro:1}
Let $G$ be a topological Abelian group such that every countable
subgroup of $G$ is $h$-embedded. Then the countable subgroups of $G$
are closed, the compact subsets of $G$ are finite, and $G$ is
sequentially closed in its completion $\varrho{G}$.
\end{proposition}

\begin{proof}
Let $K$ be a countable subgroup of $G$ and $x\in G\setminus K$.
Denote by $K_x$ the subgroup of $G$ generated by the set
$K\cup\{x\}$. The group $K_x$ is countable and $h$-embedded in $G$.
Since $K_x$ is Abelian, there exists a homomorphism $h\colon
K_x\to{\mathbb T}$ such that $h(x)\neq 1$ and $h(K)=\{1\}$. Let $f$
be a continuous extension of $h$ to a homomorphism of $G$ to
${\mathbb T}$. Then $U=f^{-1}({\mathbb T} \setminus\{1\})$ is an
open neighborhood of $x$ in $G$ and $U\cap K=\emptyset$, i.~e.,
$x\notin\overline{K}$. This proves that $K$ is closed in $G$.

Let us show that all compact subsets of $G$ are finite. Suppose to
the contrary that $C$ is an infinite compact subset of $G$. Take a
countable infinite subset $S$ of $C$ and consider the subgroup
$K=\langle S\rangle$ of $G$ generated by $S$. Then $K$ is a
countable subgroup of $G$ and, as we have just shown, $K$ is closed
in $G$. Hence $P=C\cap K$ is a closed subset of $C$, whence it
follows that $P$ is compact. Since $S\subset P\subset K$, the set
$P$ is countable and infinite. Therefore, $P$ is metrizable and
contains non-trivial convergent sequences. However, since $K$ is
$h$-embedded in $G$, the topology that $K$ inherits from $G$ is
finer than the largest precompact topology $\tau_b(K)$ on the
abstract group $K$. It remains to recall a well-known fact that an
Abelian group with the Bohr topology does not contain non-trivial
convergent sequences (see \cite{Lep} or \cite[Theorem~9.9.30]{AT}).
Clearly, this contradicts the inclusion $P\subset K$. We have thus
proved the second part of the proposition.

Finally, suppose that a sequence $\{x_n: n\in\omega\}$ of elements of
$G$ converges to an element $y\in \varrho{G}\setminus G$. Then the
nontrivial sequence $\{x_{n+1}-x_n: n\in\omega\}$ converges to the
neutral element of $G$, which is impossible since $G$ does not
contain infinite compact subsets. Hence $G$ is sequentially closed
in $\varrho{G}$.
\end{proof}

The following important fact was proved by Raczkowski and Trigos in
\cite[3.1]{RT}.

\begin{lemma}\label{racztrigos}
If $G$ is a precompact Abelian group, then the
evaluation mapping $\alpha_G$ is a topological isomorphism of $G$
onto $((G^\wedge,\tau_p)^\wedge,\tau_p).$
\end{lemma}

\begin{theorem}\label{fincompreflex}
If every compact subset of a pseudocompact Abelian group $G$ is
finite, then $G$ is reflexive.
\end{theorem}

\begin{proof}
It follows from \cite[Proposition~4.4]{HM} that for any
pseudocompact group $G$, the compact subsets of $(G^\wedge,\tau_p)$
are finite. In particular,
\[
\left((G^\wedge,\tau_p)^\wedge, \tau_{co}\right)\cong
\left((G^\wedge,\tau_p)^\wedge,\tau_p\right).
\]
Since all compact subsets of $G$ are finite, we have that
$\tau_{co}=\tau_p$ on $G^\wedge$. Using Lemma \ref{racztrigos} we
conclude that
\[
\left((G^\wedge,\tau_{co})^\wedge, \tau_{co}\right)\cong
\left((G^\wedge,\tau_p)^\wedge,\tau_p\right)\cong G,
\]
i.~e., $G$ is reflexive.
\end{proof}

\begin{remark}
The condition in Theorem~\ref{fincompreflex} that $G$ does not
contain infinite compact subsets is sufficient but not necessary.
Indeed, given  a pseudocompact noncompact reflexive group $P$ and
an infinite compact group $K$, the product group $P\times K$ is
pseudocompact, noncompact, reflexive, and contains an infinite
compact subset homeomorphic to $K$.
\end{remark}

The following fact is an immediate consequence of
Proposition~\ref{Pro:1} and Theorem~\ref{fincompreflex}.

\begin{theorem}\label{Phreflexive}
If $G$ is a pseudocompact Abelian group such that every countable
subgroup of $G$ is $h$-embedded, then $G$ is reflexive.
\end{theorem}

The following result, proved in \cite{Tk1}, will provide a first
example of a noncompact group with the above properties and will be
applied in our further arguments. Below, we denote by ${\mathfrak
c}$ the power of the continuum, so that ${\mathfrak c}=2^\omega$.
\begin{theorem}\label{examplePhZ2}
There exists a dense pseudocompact subgroup $G$ of the product group
${\mathbb Z}(2)^{\mathfrak c}$ such that all countable subgroups of $G$ are
$h$-embedded in $G$.
\end{theorem}
%

\begin{remark}\label{preprintGM}
There are many other examples of pseudocompact groups all whose
countable subgroups are $h$-embedded. Answering a question posed by
D.~Dikranjan and D.~Shakhmatov in \cite{DD}, J.~Galindo and
S.~Macario (\cite{GM}) have recently found conditions under which a
pseudocompact group admits another pseudocompact group topology
which renders all countable subgroups $h$-embedded. This reference,
which recently came to the authors' attention, contains several
generalizations of Theorem~\ref{Phreflexive}.
\end{remark}

The following two results provide more information regarding the
duals of precompact and pseudocompact groups. We start with
considering pseudocompact groups:

\begin{lemma}\label{Le:N}
The following hold true for a pseudocompact Abelian group $G$:
\begin{enumerate}
\item[a)] All countable subgroups of $G^\wedge$ are $h$-embedded
in $G^\wedge$.
\item[b)] Every compact subset of $G^\wedge$ is finite.
\end{enumerate}
\end{lemma}

\begin{proof}
a) Consider a countable subgroup $H$ of $G^{\wedge}$ and an
arbitrary homomorphism $f\colon H\to{\mathbb T}$. Since $G$ is
pseudocompact, $\omega(H,G)$ is the maximal precompact topology on
$H$ (see \cite[Proposition~3.4]{HM}). This implies, by
Comfort--Ross' theorem in \cite{CR1}, that there exists an element
$g\in G$ with $\alpha_G(g)\upharpoonright_H=f$. In particular,
$\alpha_G(g)$ is a continuous character on $G^\wedge$ which extends
$f$, and $H$ is $h$-embedded in $G^\wedge$.

b) By a), all countable subgroups of $G^\wedge$ are $h$-embedded.
Hence Proposition~\ref{Pro:1} implies that every compact subset of
$G^\wedge$ is finite. Alternatively, one can deduce the conclusion
from \cite[Proposition~4.4]{HM}.
\end{proof}
In the case of a reflexive group $G$, the precompactness and
pseudocompactness of $G$ can be completely characterized by
corresponding properties of $G^\wedge$:
\begin{proposition}\label{PR}
Let $G$ be a reflexive group. Then:
\begin{itemize}
\item[a)] $G$ is precompact iff the compact subsets
          of $G^\wedge$ are finite.
\item[b)] $G$ is pseudocompact iff the countable subgroups of
$G^\wedge$ are $h$-embedded.
\end{itemize}
\end{proposition}

\begin{proof}
(a) If $(G, {\tau})$ is precompact, then $\tau$ is the topology of
pointwise convergence on elements of $G^\wedge$, that is,
$\tau=\omega (G,G^{\wedge})$. Since every compact set $K\subset
G^\wedge$ is equicontinuous, there exists a finite set $F\subset
G^\wedge$ with $K \subset (F^\triangleleft)^{\triangleright}$. Since
$F$ is finite, so is $(F^\triangleright)^{\triangleleft}$ (see
\cite[7.11]{Au}). The reflexivity of $G$ yields that
$(F^\triangleright)^{\triangleleft}=
(F^\triangleleft)^{\triangleright}$ and, hence, $K$ is finite.

Conversely, suppose that all compact subsets of $G^\wedge$ are
finite. Then $(G^{\wedge\wedge},\tau_{p})\cong (G^{\wedge\wedge},
\tau_{co})\cong G$.

(b) The necessity follows from a) of Lemma~\ref{Le:N}. Conversely,
suppose that all countable subgroups of $G^\wedge$ are $h$-embedded.
By Proposition~\ref{Pro:1}, the compact subsets of $G^\wedge$ are
finite, so $G^{\wedge\wedge}$ carries the topology of pointwise
convergence on elements of $G^\wedge$. In other words, the
compact-open topology of $G^{\wedge\wedge}$ coincides with
$\omega(G^{\wedge\wedge},G^\wedge)$. Since $\alpha_G\colon G\to
G^{\wedge\wedge}$ is a topological isomorphism, we conclude that the
original topology of $G$ is
$\omega(G,G^\wedge)=\alpha_G^{-1}(\omega(G^{\wedge\wedge},G^\wedge))$.
Hence it suffices to show that the topology $\omega(G,G^\wedge)$ is
pseudocompact.

Take any countable subgroup $H$ of $G^\wedge$. We claim that
$\omega(H,G)$ coincides with the maximal precompact topology on the
abstract group $H$. Indeed, let $f\colon H\to{\mathbb T}$ be a homomorphism.
The subgroup $H$ is $h$-embedded in $G^\wedge$, so there exists a
continuous homomorphism $\tilde{f}\colon G^\wedge\to{\mathbb T}$ extending
$f$. Since the evaluation mapping $\alpha_G$ is surjective, there is
an element $g\in G$ with $\alpha_G(g)=\tilde{f}$. This proves our
claim. Hence \cite[Proposition~3.4]{HM} implies the pseudocompactness
of $(G,\omega(G,G^\wedge))$.
\end{proof}

It seems strange at the first sight, but Proposition~\ref{Pro:1} and
Theorem~\ref{fincompreflex} enable us to find examples of reflexive
precompact non-pseudo\-com\-pact groups. Our construction makes use
of the following result close to Proposition~\ref{Pro:1}. As usual,
we denote by $\langle x\rangle$ the cyclic subgroup of a group $L$
generated by an element $x\in L$.

\begin{lemma}\label{Le:Q}
Let $G$ be any infinite pseudocompact Boolean group all countable
subgroups of which are $h$-embedded (Theorem~\ref{examplePhZ2}).
Suppose that $C_0$ is a countable infinite subgroup  of $G$.
Consider $K=\langle x_0\rangle$ for some $x_0\in\overline{C}_0
\setminus C_0$ (the closure of $C_0$ is taken in the completion
$\varrho{G}$ of $G$) and the quotient homomorphism $\pi\colon
\varrho{G}\to\varrho{G}/K$. Then all compact subsets of the subgroup
$H=\pi(G)\subset\varrho{G}/K$ are finite and $H$ is reflexive.
\end{lemma}

\begin{proof}
Since the group $C_0$ is infinite, its closure $\overline{C}_0$ is
an infinite compact group and, hence, $|\overline{C}_0|\geq
2^\omega$. It follows from Proposition~\ref{Pro:1} that $C_0$ is
closed in $G$, so $\overline{C}_0\cap G=C_0$. Hence
$\overline{C}_0\setminus C_0$ is an uncountable subset of
$\varrho{G}\setminus G$. This explains, in particular, why we can
find an element $x_0\in\overline{C}_0\setminus C_0$.

We claim that for every countable subgroup $C$ of $G$ with $C_0
\subset C$, the image $\pi(C)$ is closed in $H$. Since the
homomorphism $\pi$ is open, it suffices to verify that
$\pi^{-1}\pi(C)=C+K$ is closed in $\pi^{-1}\pi(G)=G+K$. If not,
there exists $y\in G+K$ such that $y\in\overline{C+K}\setminus
(C+K)$. There are two possibilities:

Case~I. $y\in G$. Since $C$ is closed in $G$, we deduce that
$y\notin\overline{C}$. It follows from $\overline{C+K}=\overline{C}
\cup\overline{C+x_0}$ that $y\in\overline{C+x_0}=\overline{C}+x_0$.
In its turn, this implies that $x_0\in\overline{C}+y$. Since
$C_0\subset C$, we have that $x_0\in \overline{C}\cap
(\overline{C}+y)$ and, hence, $y\in\overline{C}$. This is a
contradiction.

Case~II. $y\in G+x_0$. Then the element $z=y+x_0\in G+K$ satisfies
$z\in\overline{C+K}\setminus(C+K)$. We have just shown in Case~I,
however, that this is impossible. This proves that $\pi(C)$ is
closed in $H$.

Let us see that all compact subsets of the subgroup $H$ are finite.
Suppose to the contrary that $F$ is an infinite compact subset of
$H$. There exists a countable subgroup $C$ of $G$ such that the set
$\pi(C)\cap F$ is infinite and $C_0\subset C$. We have just proved that
$\pi(C)$ is a closed subgroup of $H$. Therefore, $P=\pi(C)\cap F$ is
a countable, infinite, compact subset of $H$. Since the projection
$\pi\colon\varrho{G}\to\varrho{G}/K$ is perfect, $Q=\pi^{-1}(P)$ is
a countable infinite compact subset of $\varrho{G}$.

Let $R=Q\cap G$. Clearly, $R\subset C$ and $P=\pi(R)$. It follows from
the definition of $Q$ and $R$ that
\[Q=R+K = R\cup (R+x_0).
\]
Since $Q$ is countable, infinite, and compact, it contains a
nontrivial convergent sequence. By Proposition~\ref{Pro:1}, the
group $G$ and its subspace $R$ do not contain infinite compact
subsets. The same is also valid for the subspace $R+x_0$ of
$\varrho{G}$. Therefore, since the group $\varrho{G}$ is Boolean and
$Q=R\cup (R+x_0)$, there exists a sequence $D=\{x_n: n\in\omega\}
\subset R$ converging to an element $y\in Q\setminus R=Q\setminus
(Q\cap G)$. This contradicts the sequential completeness of $G$ (see
Proposition~\ref{Pro:1}) and implies that all compact subsets of $H$
are finite.

Finally, the group $H$ is pseudocompact as a continuous image of the
pseudocompact group $G$. Hence the reflexivity of $H$ follows from
Theorem~\ref{fincompreflex}.
\end{proof}

In the following theorem we extend the frontiers of the class of
reflexive precompact groups.

\begin{theorem}\label{Th:4}
There exist precompact reflexive groups which are not pseudocompact.
\end{theorem}

\begin{proof}
Let $G$ be any infinite pseudocompact Boolean group all countable
subgroups of which are $h$-embedded (Theorem~\ref{examplePhZ2}).
Take a countable infinite subgroup $C$ of $G$ and pick a point
$x_0\in\overline{C}\setminus C$, where the closure is taken in the
completion $\varrho{G}$ of $G$. Then the group $K=\langle x_0
\rangle$ has a trivial intersection with $G$. Let $\pi\colon
\varrho{G}\to\varrho{G}/K$ be the canonical projection. By
Lemma~\ref{Le:Q}, the subgroup $H=\pi(G)$ of the compact group
$\varrho{G}/K$ is reflexive. Hence its dual $D=H^\wedge$ is
reflexive as well.

We claim that the group $D$ is not pseudocompact. Indeed, otherwise
item a) of Lemma~\ref{Le:N} would imply that all countable subgroups
of the dual group $D^\wedge\cong H$ were $h$-embedded in $H$.
However, our choice of $x_0\in\overline{C}\setminus C$ together with
the fact that the intersection $K\cap G$ is trivial imply that the
topology of $\pi(C)$ is strictly coarser than the topology of $C$
or, in other words, the restriction of $\pi$ to $C$ is not open when
considered as a mapping of $C$ onto $\pi(C)$. Indeed, by
\cite[1.3]{grant}, this restriction is open if and only if the
intersection $K\cap C$ is dense in $K$, which is not our case. Thus
the maximal precompact topology on the abstract group $\pi(C)$ is
strictly finer than the topology of $\pi(C)$ inherited from
$H=\pi(G)$ and, hence, $\pi(C)$ is not $h$-embedded in $H$. This
finishes the proof.
\end{proof}

According to Theorem \ref{Phreflexive} and
Theorem~\ref{examplePhZ2}, there are reflexive groups which are
pseudocompact noncompact. By Theorem~\ref{Th:4}, there are reflexive
groups which are precompact non-pseudocompact. It is natural to ask,
therefore, if there exist countably compact noncompact reflexive
groups. We show in Proposition~\ref{Exa:HJ} below that,
consistently, the answer is \lq\lq{yes\rq\rq}.

\begin{proposition}\label{Exa:HJ}
Under Martin's Axiom, every Boolean group $G$ of cardinality
${\mathfrak c}$ admits a reflexive countably compact noncompact
topological group topology.
\end{proposition}

\begin{proof}
Let $G$ be a Boolean group with $|G|={\mathfrak c}$. It was
established in the proof of \cite[Theorem~3.9]{DT2} that, under the
assumption of Martin's Axiom, there exists an (abstract)
monomorphism $h\colon G\to{\mathbb Z}(2)^{\mathfrak c}$ satisfying
the following conditions:
\begin{enumerate}
\item[(a)] every infinite subset $S$ of $h(G)$ is \textit{finally
dense\/} in ${\mathbb Z}(2)^{\mathfrak c}$, i.~e., there exists
$\alpha<{\mathfrak c}$ such that $\pi_{{\mathfrak
c}\setminus\alpha}(S)$ is dense in ${\mathbb Z}(2)^{{\mathfrak
c}\setminus\alpha}$, where $\pi_A\colon{\mathbb Z}(2)^{\mathfrak
c}\to{\mathbb Z}(2)^A$ is the projection for any set
$A\subset{\mathfrak c}$;
\item[(b)] $\pi_\alpha(h(G))={\mathbb Z}(2)^\alpha$, for every
$\alpha<{\mathfrak c}$.
\end{enumerate}
Let
\[
{\mathcal T}=\{h^{-1}(U\cap h(G)): U\ \mbox{is\ open\ in}\
{\mathbb Z}(2)^{\mathfrak c}\}.
\]
Then ${\mathcal T}$ is a Hausdorff topological group topology on $G$
and $h$ is a topological isomorphism of $H=(G,{\mathcal T})$ onto
the subgroup $h(G)$ of ${\mathbb Z}(2)^{\mathfrak c}$. It was proved
in \cite[Theorem~3.9]{DT2}, using (a) and (b), that the group $H$ is
countably compact.

We claim that all compact subsets of $H$ are finite. Suppose to the
contrary that $F$ is an infinite compact subset of $H$. According to
(a), there exists $\alpha<{\mathfrak c}$ such that $\pi_{{\mathfrak
c}\setminus\alpha}(F)$ is dense in ${\mathbb Z}(2)^{{\mathfrak
c}\setminus\alpha}$. Since $F$ is compact, we have that
$\pi_{{\mathfrak c}\setminus\alpha}(F)={\mathbb Z}(2)^{{\mathfrak
c}\setminus\alpha}$. The latter is impossible since
$|F|\leq{\mathfrak c}$ while $|{\mathbb Z}(2)^{{\mathfrak
c}\setminus\alpha}|= |{\mathbb Z}(2)^{\mathfrak c}|=2^{\mathfrak
c}>{\mathfrak c}$.

Thus, $H$ is a countably compact, noncompact Abelian group without
infinite compact subsets. The reflexivity of $H$ follows from
Theorem~\ref{fincompreflex}.
\end{proof}

The countably compact group $H=(G,{\mathcal T})$ in
Proposition~\ref{Exa:HJ} does not contain infinite compact subsets.
It is an open problem still whether there exist in $ZFC$ infinite
countably compact groups without non-trivial convergent sequences.
This makes it interesting to find out whether
countably compact noncompact reflexive groups can exist in $ZFC$
alone. Very recently this question was solved in the positive in
\cite{GRT}.

\section{Some properties of the class $\mathcal{P}_h$}\label{Sec:ClP}
Recall that $\mathcal{P}_h$ is the class of pseudocompact Abelian
groups whose countable subgroups are $h$-embedded. We proved in
Theorem~\ref{Phreflexive} that the groups in this class are
reflexive, and we will now establish some other properties of this
class concerning the duality theory.

\begin{theorem}\label{Th:PPro}
If a topological group $G$ is in $\mathcal{P}_h$, so is $G^\wedge$.
\end{theorem}

\begin{proof}
Let $G\in\mathcal{P}_h$. Then $G$ is reflexive, by
Theorem~\ref{Phreflexive}. Hence, according to item b) of
Proposition~\ref{PR}, the topology
$\omega(G^\wedge,G^{\wedge\wedge}) = \omega(G^\wedge,G)$ is
pseudocompact, i.~e., the group $G^\wedge$ is pseudocompact. The
countable subgroups of $G^\wedge$ are $h$-embedded in $G^\wedge$, by
Lemma~\ref{Le:N}.
\end{proof}

The following concept was introduced in \cite{HM}. A topological
group $G$ is called \textit{countably pseudocompact\/} if for every
countable set $A\subset G$, there exists a countable set $B\subset
G$ such that $A\subset \overline{B}$ and $\overline{B}$ is
pseudocompact. It is easy to see that every countably pseudocompact
group is pseudocompact, but not vice versa. Hern\'andez and Macario
proved in \cite[Corollary~4.3]{HM} that for every countably
pseudocompact Abelian group $G$, the dual group $G^\wedge$ endowed
with the topology of pointwise convergence on elements of $G$ is a
$\mu$-space, i.~e., the closure of every functionally bounded subset
of $(G^\wedge,\tau_p)$ is compact. In the same article they also
raised the problem as to whether the conclusion remained valid for
pseudocompact groups. We solve the problem in the negative:

\begin{corollary}\label{Cor:HM}
There exists a pseudocompact Abelian group $G$ such that the dual
group $G^\wedge$ is pseudocompact and noncompact. Hence $G^\wedge$
fails to be a $\mu$-space.
\end{corollary}

\begin{proof}
Take any infinite group $G\in\mathcal{P}_h$ (see
Theorem~\ref{examplePhZ2}). By Theorem~\ref{Th:PPro}, the dual group
$G^\wedge$ is in $\mathcal{P}_h$ as well, so $G^\wedge$ is
pseudocompact. Clearly, $G^\wedge$ is not compact --- otherwise the
bidual group $G^{\wedge\wedge}$ would be discrete. The latter is
impossible since, by Theorem~\ref{Phreflexive}, the groups
$G^{\wedge\wedge}$ and $G$ are topologically isomorphic.
\end{proof}

Finally, we establish that the class $\mathcal{P}_h$ is wide and
that the operation of taking quotient groups completely destroys
reflexivity, even if the kernels of the respective quotient
homomorphisms are pseudocompact.

\begin{theorem}\label{Th:QG}
If $\,G$ is a pseudocompact Abelian group, there is a reflexive
group $H\in\mathcal{P}_h$ such that $G\cong H/L$, with $L$ a closed
 reflexive pseudocompact subgroup of $H$.
\end{theorem}

\begin{proof}
By \cite[Theorem~5.5]{DT}, for every pseudocompact Abelian group
$G$, one can find a pseudocompact sequentially complete Abelian
group $H$ and a closed pseudocompact subgroup $L$ of $H$ such that
the group $G$ is topologically isomorphic to the quotient group
$H/L$.

Sequential completeness of $H$ follows from a stronger property of
$H$, namely, its countable subgroups are $h$-embedded. This was
explicitly shown in the proof of \cite[Theorem~5.5]{DT}. Hence
$H\in\mathcal{P}_h$. Finally, the group $H$ and its closed
pseudocompact subgroup $L$ are reflexive according to
Theorem~\ref{Phreflexive}.
\end{proof}

 It is worth comparing the above result with
Theorem~2.6 from \cite{BCM} saying that taking quotients with
respect to compact subgroups preserves reflexivity.

{\bf Acknowledgements.}
This work was done while the last listed author was visiting Spain
in 2008. He thanks his hosts for generous support and hospitality.

\newpage
  \noindent S.~Ardanza-Trevijano and M.~J.~Chasco\\
   Departamento de F\'{\i}sica y Matem\'atica Aplicada\\
   Universidad de Navarra\\
   Irunlarrea s.~n.\\
   31080 Pamplona, Spain \\
   \verb+sardanza@unav.es+ \\ \verb+mjchasco@unav.es+
   
\vspace{.5cm}

 \noindent  X.~Dom\'{\i}nguez\\
   Departamento de M\'{e}todos Matem\'{a}ticos y de
Representaci\'{o}n\\
   Universidad de A Coru\~na\\
   Campus de Elvi\~na s.~n.\\
   15071 A Coru\~na, Spain\\
   \verb+xdominguez@udc.es+

\vspace{.5cm}

 \noindent 
   M.~Tkachenko\\
  Departamento de Matem\'aticas\\
   Universidad Aut\'onoma Metropolitana\\
   Av.~San Rafael Atlixco 186\\
   Col.~Vicentina, Iztapalapa, C.~P. 09340\\
   M\'exico, D.~F.\\
   \verb+mich@xanum.uam.mx+
\end{document}